\documentclass{amsart}

\usepackage[margin=1.5in]{geometry}
\usepackage{amsmath}
\usepackage{amssymb, amsthm, mathrsfs}
\usepackage{hyperref}

\newtheorem{theorem}{Theorem}[section]

\newtheorem{lemma}[theorem]{Lemma}
\newtheorem{proposition}[theorem]{Proposition}
\newtheorem{corollary}[theorem]{Corollary}

\theoremstyle{definition}
\newtheorem{definition}[theorem]{Definition}

\theoremstyle{remark}
\newtheorem{remark}[theorem]{Remark}

\numberwithin{equation}{section}

\newcommand{\abs}[1]{\left|#1\right|}
\newcommand{\tr}{\textup{tr}}

\newcommand{\Ric}{\textup{Ric}}
\newcommand{\Div}{\textup{div}}

\newcommand{\Id}{\textup{Id}}

\newcommand{\eval}[2]{\left. #1 \right|_{#2}}

\newcommand{\R}{\mathbb{R}}
\newcommand{\C}{\mathbb{C}}

\newcommand{\K}{\mathcal{K}}

\renewcommand{\DH}{\mathcal{DH}}
\renewcommand{\H}{\mathcal{H}}

\newcommand{\ddbar}{\partial \bar{\partial}}

\begin{document}
\title{Boltzmann's Entropy and K\"ahler-Ricci Solitons}

\author{Frederick Tsz-Ho Fong}
\address{Department of Mathematics, School of Science, Hong Kong University of Science and Technology, Clear Water Bay, Hong Kong}
\email{frederick.fong@ust.hk}

\date{May 26, 2016}
\begin{abstract}
We study a Boltzmann's type entropy functional (which appeared in existing literature) defined on K\"ahler metrics of a fixed K\"ahler class. The critical points of this functional are gradient K\"ahler-Ricci solitons, and the functional was known to be monotonically increasing along the K\"ahler-Ricci flow in the canonical class.

In this article, we derive and analyze the second variation formula for this entropy functional, and show that all gradient K\"ahler-Ricci solitons are stable with respect to this entropy functional. Furthermore, using this result, we give a new proof that gradient shrinking K\"ahler-Ricci solitons are stable with respect to the Perelman's entropy in a fixed K\"ahler class.
\end{abstract}
\maketitle
\section{Introduction}
In this article, we examine an entropy functional defined on the space of K\"ahler potentials of a compact K\"ahler manifold whose first Chern class has a definite sign. This functional, which will be denoted by $\H$ in this article, appeared in some existing literature related to the K\"ahler-Ricci flow and solitons.

Let $X^n$ be a compact K\"ahler manifold whose first Chern class $c_1(X)$ has a definite sign denoted by $\lambda \in \R$. Let $\omega_0$ be a K\"ahler metric such that $\lambda \omega_0 \in c_1(X)$. Denote the space of K\"ahler potentials by:
\[\K = \{\varphi \in C^{\infty}(X, \R) : \omega_0 + \sqrt{-1}\ddbar\varphi > 0\}.\]
Given any $\varphi \in \K$, we denote $\omega_\varphi := \omega_0 + \sqrt{-1}\ddbar\varphi$. Since $\lambda[\omega_\varphi] = \lambda[\omega_0] = c_1(X) = [\Ric(\omega_\varphi)]$, the $\ddbar$-lemma asserts that there exists a unique smooth function $f_{\varphi}$ on $X$ such that:
\begin{align*}
\lambda\omega_\varphi - \Ric(\omega_\varphi) & = \sqrt{-1}\ddbar f_\varphi; \text{ and }\\
\int_X e^{-f_\varphi} \omega_\varphi^n & = \int_X \omega_\varphi^n.
\end{align*}
Note that $\omega_\varphi$ is in the same K\"ahler class as $\omega_0$, and therefore the total volume
\[[\omega_\varphi]^n := \int_X\omega_\varphi^n\]
is independent of $\varphi \in \K$. We define the entropy functional $\H : \K \to \R$ by:
\begin{equation}
\label{eq:H}
\H(\varphi) := \frac{1}{[\omega_0]^n}\int_X f_\varphi e^{-f_{\varphi}} \omega_\varphi^n.
\end{equation}

This $\H$-functional can be expressed as a special form of Boltzmann's entropy in statistical thermodynamics. We first rewrite $f_\varphi$ as:
\begin{align*}
f_\varphi & = -\log\left(e^{-f_\varphi}\right) = -\log\left(\frac{e^{-f_\varphi}\omega_\varphi^n}{\omega_\varphi^n}\right)\\
& = -\log\left(\frac{e^{-f_\varphi}\omega_\varphi^n}{[\omega_\varphi]^n} \Bigg/ \frac{\omega_\varphi^n}{[\omega_\varphi]^n}\right)
\end{align*}
Denote $d\nu_\varphi := \displaystyle{\frac{e^{-f_\varphi}\omega_\varphi^n}{[\omega_\varphi]^n}}$ and $d\mu_\varphi := \displaystyle{\frac{\omega_\varphi^n}{[\omega_\varphi]^n}}$. Note that both $d\nu_\varphi$ and $d\mu_\varphi$ are probability measures on $X$. Under these notations, the $\H$-functional can be written as:
\[\H(\varphi) = -\int_X \frac{d\nu_\varphi}{d\mu_\varphi} \log\left(\frac{d\nu_\varphi}{d\mu_\varphi}\right)d\mu_\varphi = -\int_X \log\left(\frac{d\nu_\varphi}{d\mu_\varphi}\right)d\nu_\varphi\]
which is the Boltzmann's entropy with respect to the two measures $d\nu_\varphi$ and $d\mu_\varphi$.

On a compact Riemannian manifold $(M,g)$, Lott and Villani studied in \cite{LV} the Boltzmann's entropy (different from $\H$ in this article) of the form:
\[H_{d\mu_0}(d\nu_t) := \int \frac{d\nu_t}{d\mu_0}\log\left(\frac{d\nu_t}{d\mu_0}\right)d\mu_0,\]
where $d\mu_0$ is a fixed measure and $d\nu_t$ is a \emph{geodesic} path of measures (absolutely continuous with respect to $d\mu_0$) in the Wasserstein space $\mathcal{P}_2(M)$. They showed that this entropy is convex (i.e. $\frac{d^2}{dt^2}H_{d\mu_0}(d\nu_t) \geq 0$) for any geodesic paths $d\nu_t$ if and only if $(M, g)$ has non-negative Ricci curvature.

Concerning our $\H$-functional defined in \eqref{eq:H}, there are several interesting results about it in the subject of the K\"ahler-Ricci flow. In \cite[Section 6]{PSSW} (see also \cite{He}), it was proved that the $\H$-functional\footnote{In \cite{PSSW}, the letter $b$ was used to denote the $\H$-functional. In \cite{He}, the author adopted a different sign convention for this functional.} is monotonically increasing along the normalized K\"ahler-Ricci flow $\frac{\partial\omega}{\partial t} = -\Ric(\omega) + \lambda\omega$ for the case $c_1(X) > 0$ assuming the initial metric $\omega_0 \in \frac{1}{\lambda}c_1(X)$. In fact it is also true in the cases of $c_1(X) = 0$ and $c_1(X) < 0$ (see the discussion in P.\pageref{KRFMonotone} in this article). Furthermore, the critical points of this functional are K\"ahler potentials $\varphi$ such that $\omega_\varphi$ is a gradient K\"ahler-Ricci soliton, meaning that:
\[\Ric(\omega_\varphi) + \nabla^2 f_{\varphi} = \lambda\omega_\varphi\]
and so $\nabla f_\varphi$ is a real holomorphic vector field. In particular, if $f_\varphi$ is a constant function, then $\omega_\varphi$ is a K\"ahler-Einstein metric.

This $\H$-functional also plays a role in the unpublished result due to Perelman (see \cite{SeT, Cao13}) in which he proved that the diameter and scalar curvature are uniformly bounded along the normalized K\"ahler-Ricci flow $\frac{\partial\omega}{\partial t} = -\Ric(\omega) + \lambda\omega$ starting with $\omega_0 \in \frac{1}{\lambda} c_1(X) > 0$. In Section 2 of \cite{SeT} and Section 6 of \cite{Cao13}, monotonicity of Perelman's $\mathcal{W}$-functional was used to show that the Ricci potential $f_\varphi$ is uniformly bounded from below along the normalized K\"ahler-Ricci flow. In fact, a uniform bound for $\H$ is sufficient to prove a uniform lower bound for $f_\varphi$ along the flow. The lower bound of $\H$ follows from the monotonicity along the K\"ahler-Ricci flow, and the upper bound can be deduced using Jensen's inequality. Both are easier to obtain than the analogous results of the $\mathcal{W}$-functional. 

In \cite[Proposition 2.2]{He}, the author derived an upper bound for $\H$ in relation to the maximal compact subgroup of $\text{Aut}(M,\omega_0)$. Using this result, it was proved in \cite[Corollary 2.6]{He} that if a K\"ahler-Ricci soliton exists, then it maximizes Perelman's $\mu$-functional.

In the recent article \cite{D15}, Donaldson also pointed out that the $\H$-functional and the Ding's functional $\mathcal{F}$ introduced in \cite{Ding88} are related by $\frac{\partial\mathcal{F}}{\partial t} \leq \H$ along the normalized K\"ahler-Ricci flow starting from $\omega_0 \in \frac{1}{\lambda}c_1(X)$.

We are going to explore this $\H$-functional further in this article by deriving and analyzing the \emph{second variation formula} for $\H$. For the Perelman's $\nu$-functional introduced in \cite{P}, the Euler-Lagrange's equation gives gradient shrinking Ricci solitons as its critical points. The second variation of $\nu$ was discussed and computed in \cite{CHI, CM} and a notion of stability of shrinking Ricci solitons was developed using the second variation formula. Various works concerning about stability of shrinking Ricci solitons can be found in e.g. \cite{HM11, HM13, Kr}. Analogously, from the second variation formula for $\H$, we introduce a stability operator $\mathcal{S}_f$ and a notion of $\H$-stability for gradient K\"ahler-Ricci solitons. Our main result is that any critical point (i.e. gradient K\"ahler-Ricci solitons) is $\H$-stable:

\subsection*{Main Theorem}
Let $\varphi$ be a critical point of the functional $\H$ (so that $\omega_\varphi$ is a gradient K\"ahler-Ricci soliton), then for any $\psi \in T_\varphi\K$, we have:
\[\left.\frac{d^2}{dt^2}\right|_{t = 0}\H(\varphi + t\psi) \leq 0\]
and equality holds if and only if $\nabla \psi$ is a real holomorphic vector field.\qed

\vskip 0.5cm

From dynamical system viewpoint, the main theorem and the monotonicity of $\H$ along the K\"ahler-Ricci flow assert that gradient K\"ahler-Ricci solitons are ``attractors'', and $-\H$ acts as a Lyapunov function for the flow. It is well-known in \cite{Cao85} that the normalized K\"ahler-Ricci flow in the canonical class converges to K\"ahler-Einstein metrics when $c_1(X) = 0$ or $c_1(X) < 0$. In case of $c_1(X) > 0$, assuming the existence of a shrinking K\"ahler-Ricci soliton, the normalized K\"ahler-Ricci flow in the canonical class was shown in \cite{TZ07} to converge to the soliton under some invariant condition on the initial metric (see also \cite{PS06, PSSW08, Zh11, TZ13, TZZZ}). It is hoped that the main result of this article could bring more insight about the stability of the K\"ahler-Ricci flow when it approaches to the soliton limit.

In the case of $c_1(X) > 0$, the $\H$-functional is also related to the Perelman's $\mu$-functional, in a sense that $\H$ is an upper barrier of $\mu$ (up to addition of a constant) and that they coincide at gradient K\"ahler-Ricci solitons (again up to addition of a constant). Therefore, the main theorem of this article implies (see Proposition \ref{prop:PerelmanStability}) an earlier result in \cite{TZ08} that K\"ahler-Ricci solitons are $\mu$-stable in a fixed K\"ahler class, i.e. $\mu$-stable in the direction of complex Hessian of potential functions. This is an important result used in many works about the stability of K\"ahler-Ricci solitons and the convergence of the K\"ahler-Ricci flow (see e.g. \cite{TZ08, SW,Zhe,Z13}).

\subsection*{Acknowledgement}
The author would like to thank Zhou Zhang and Kai Zheng for some enlightening discussions and for their interests in this work. The author is supported by HKUST New-Faculty Initiation Grant IGN15SC04 and HKUST Start-up Grant R9353.

\newpage
\section{Ricci Potential and its Variation}
In this section, we study the function $f_\varphi$ (commonly called the Ricci potential of $\omega_\varphi$) and derive the evolution equation for $f_\varphi$ which will be used often later on. To begin, we express the Ricci potential in a more explicit way in terms of the Ricci potential of $\omega_0$:
\begin{align*}
\lambda\omega_\varphi - \Ric(\omega_\varphi) & = \sqrt{-1}\ddbar f_\varphi & \Longrightarrow & \quad \lambda\omega_0 + \lambda \sqrt{-1}\ddbar \varphi + \sqrt{-1}\ddbar \log\omega_\varphi^n = \sqrt{-1}\ddbar f_\varphi\\
\lambda\omega_0 - \Ric(\omega_0) & = \sqrt{-1}\ddbar f_0 & \Longrightarrow & \quad \lambda\omega_0 + \sqrt{-1}\ddbar\log\omega_0^n = \sqrt{-1}\ddbar f_0
\end{align*}
and so by subtraction, we have:
\[\sqrt{-1}\ddbar \left(\lambda\varphi + \log\frac{\omega_\varphi^n}{\omega_0^n}\right) = \sqrt{-1}\ddbar\left(f_\varphi - f_0\right).\]
Therefore, $\displaystyle{\lambda\varphi + \log\frac{\omega^n_\varphi}{\omega^n_0} - (f_\varphi - f_0)}$ is a constant on $X$. Using the normalization conditions:
\[\int_X e^{-f_\varphi}\omega_\varphi^n = \int_X e^{-f_0}\omega_0^n = [\omega_0]^n,\]
one can determine this constant and show that the Ricci potential $f_\varphi$ is given by:\begin{equation}
\label{eq:f_varphi}
f_\varphi = f_0 + \log\frac{\omega_\varphi^n}{\omega_0^n} + \lambda\varphi + \log\left(\frac{1}{[\omega_0]^n}\int_X e^{-f_0 - \lambda\varphi}\omega_0^n\right).
\end{equation}

\begin{lemma}[Evolution Equation of $f_\varphi$]
\label{lma:VarF}
Let $\varphi(t)$ be a 1-parameter smooth family of K\"ahler potentials in $\K$, where $t \in (-\varepsilon, \varepsilon)$. Denote $\psi := \frac{\partial\varphi}{\partial t}$, then the Ricci potential $f_{\varphi(t)}$ evolves by:
\begin{equation}
\label{eq:dfdt}
\frac{\partial f_{\varphi}}{\partial t} = \Delta\psi + \lambda\left(\psi - \frac{1}{[\omega_0]^n}\int_X \psi e^{-f_\varphi} \omega_\varphi^n\right),
\end{equation}
where $\Delta := \Delta_\varphi$ is the Laplacian with respect to $\omega_{\varphi(t)}$.
\end{lemma}
\begin{proof}
Recall that $f_{\varphi}$ is defined by $\lambda\omega_{\varphi} - \Ric(\omega_{\varphi}) = \sqrt{-1}\ddbar f_{\varphi}$. Therefore, we have:
\[\lambda\left(\omega_0 + \sqrt{-1}\ddbar\varphi\right) + \sqrt{-1}\ddbar\log\det(\omega_0 + \sqrt{-1}\ddbar\varphi) = \sqrt{-1}\ddbar f_{\varphi}.\]
Differentiating both sides with respect to $t$, we get:
\[\sqrt{-1}\ddbar\left(\lambda\psi\right) + \sqrt{-1}\ddbar\left(\tr_{\omega_{\varphi}}\sqrt{-1}\ddbar\psi\right) = \sqrt{-1}\ddbar\left(\frac{\partial f_{\varphi}}{\partial t}\right).\]
Since $X$ is compact, we have:
\begin{equation}
\label{eq:dfdt_intermediate}
\lambda\psi + \Delta\psi + c(t) = \frac{\partial f_{\varphi}}{\partial t} 
\end{equation}
where $c(t)$ is a function of $t$ only to be determined. By the normalization condition on $f_\varphi$, we know:
\begin{align*}
0 & = \frac{d}{dt}\int_X e^{-f_{\varphi}}\omega_{\varphi}^n\\
& = \int_X \left(-\frac{\partial f_{\varphi}}{\partial t} + \Delta\psi\right)e^{-f_{\varphi}}\omega_{\varphi}^n
\end{align*}
where we have used the fact that $\frac{\partial}{\partial t}\log\omega_{\varphi}^n = \tr_{\omega_{\varphi}}\frac{\partial}{\partial t}\omega_{\varphi} = \tr_{\omega_{\varphi}}\sqrt{-1}\ddbar\psi = \Delta\psi$. Combining with \eqref{eq:dfdt_intermediate}, we have:
\[\int_X (\lambda\psi + c) e^{-f_{\varphi}} \omega_{\varphi}^n = 0\]
and so using the normalization condition on $f_\varphi$, we can complete the proof of the lemma by observing that:
\[c(t) = -\frac{\lambda}{[\omega_0]^n}\int_X \psi e^{-f_{\varphi}}\omega_{\varphi}^n.\]
\end{proof}

\section{Critical Points of $\H$}
The first variation of $\H$ and the Euler-Lagrange's equation have been studied in \cite{PSSW, He} when $c_1(X) > 0$, in which the critical points of $\H$ were known to be K\"ahler-Einstein metrics and more generally (shrinking) gradient K\"ahler-Ricci solitons. The cases of $c_1(X) = 0$ and $c_1(X) < 0$ can be proved in similar ways. We include the detail below for easy reference.

\begin{proposition}[First Variation of $\H$]
\label{prop:FirstVarH}
The first variation of $\H$ along a 1-parameter smooth family $\varphi(t)$ of K\"ahler potentials in $\K$ such that $\frac{\partial\varphi}{\partial t} = \psi$ is given by:
\begin{equation}
\label{eq:FirstVarH}
\frac{d}{dt}\H(\varphi) = -\frac{1}{[\omega_0]^n} \int_X \psi\,\left[\Delta f_\varphi - \abs{\nabla f_\varphi}^2 + \lambda\left(f_\varphi - \H(\varphi)\right)\right]\,e^{-f_\varphi}\omega_\varphi^n
\end{equation}
\end{proposition}
\begin{proof}
Recall that:
\[\H(\varphi) = \frac{1}{[\omega_0]^n}\int_X f_\varphi e^{-f_\varphi}\omega_\varphi^n.\]
By direct computations with the use of \eqref{eq:f_varphi}, we get:
\begin{align*}
\frac{d}{dt}\H(\varphi) & = \frac{1}{[\omega_0]^n}\int_X \left(\frac{\partial f_\varphi}{\partial t} - f_\varphi \frac{\partial f_\varphi}{\partial t} + f_\varphi \Delta\psi\right)e^{-f_\varphi} \omega_\varphi^n\\
& = \frac{1}{[\omega_0]^n}\int_X \left[\left(\Delta\psi + \lambda\left(\psi - \frac{\int_X\psi e^{-f_\varphi}\omega_\varphi^n}{[\omega_0]^n}\right)\right)\right.\\
& \qquad\qquad\qquad \left. - f_\varphi \left(\Delta\psi + \lambda\left(\psi - \frac{\int_X \psi e^{-f_\varphi}\omega_\varphi^n}{[\omega_0]^n}\right)\right) + f_\varphi\Delta\psi\right] e^{-f_\varphi}\omega_\varphi^n\\
& = \frac{1}{[\omega_0]^n} \left(\int_X \left(\Delta\psi - \lambda f_\varphi\psi\right)e^{-f_\varphi}\omega_\varphi^n + \lambda\H(\varphi)\int_X\psi e^{-f_\varphi}\omega_\varphi^n\right)
\end{align*}
Here we have used the fact that:
\[\int_X \left(\psi - \frac{\int_X\psi e^{-f_\varphi}\omega_\varphi^n}{[\omega_0]^n}\right) e^{-f_\varphi}\omega_\varphi^n = 0.\]
Through integration-by-parts, we get:
\begin{align*}
\int_X (\Delta\psi) e^{-f_\varphi}\omega_\varphi^n & = -\int_X \langle \nabla\psi, -\nabla f_\varphi\rangle e^{-f_\varphi} \omega_\varphi^n\\
& = -\int_X \psi \left(\Delta f_\varphi - \abs{\nabla f_\varphi}^2\right) e^{-f_\varphi}\omega_\varphi^n.
\end{align*}
Combining with the previous result, we have proved:
\begin{equation*}
\frac{d}{dt}\H(\varphi) = -\frac{1}{[\omega_0]^n}\int_X \psi \left[\Delta f_\varphi - \abs{\nabla f_\varphi}^2 + \lambda\left(f_\varphi - \H(\varphi)\right)\right] e^{-f_\varphi}\omega_\varphi^n
\end{equation*}
as desired.
\end{proof}

\begin{definition}[$L^2$-gradient of $\H$]
In view of Proposition \ref{prop:FirstVarH}, we denote
\[\DH(\varphi) := -\left[\Delta f_\varphi - \abs{\nabla f_\varphi}^2 + \lambda\left(f_\varphi - \H(\varphi)\right)\right]\]
which stands for the $L^2$-gradient of $\H$ in the space $\K$ with respect to the measure $\frac{1}{[\omega_0]^n}e^{-f_\varphi}\omega_\varphi^n$. Then along $\frac{\partial\varphi}{\partial t} = \psi$, we have:
\[\frac{\partial}{\partial t}\H(\varphi) = \frac{1}{[\omega_0]^n}\int_X \psi\,\DH(\varphi)\,e^{-f_\varphi}\omega_\varphi^n.\]
\end{definition}

For simplicity, we will occasionally denote $f := f_\varphi$ whenever $\varphi$ can be understood from the context. Next we introduce three linear operators on $C^\infty(X, \C)$:  given any $\psi \in C^{\infty}(X,\C)$, we define:
\begin{align*}
L_f \psi & := \Delta \psi - \nabla^i f \, \nabla_i\psi = \Delta\psi - g^{i\bar j}\,\nabla_{\bar j}f\,\nabla_i\psi,\\
\bar{L}_f \psi & := \Delta \psi - \nabla^{\bar i} f \, \nabla_{\bar i} \psi = \Delta\psi - g^{j\bar i}\,\nabla_j f\,\nabla_{\bar i} \psi,\\
\Delta_f \psi & := \Delta\psi - \frac{1}{2}\left(g^{i\bar j}\,\nabla_{\bar j}f\,\nabla_i\psi + g^{j\bar i}\,\nabla_j f\,\nabla_{\bar i} \psi\right)
\end{align*}
where $\Delta$, $\nabla$ and the inner product $\langle \cdot , \cdot \rangle$ are taken with respect to the metric $\omega_\varphi$. It is clear that $\Delta_f \psi = \frac{1}{2}\left(L_f + \bar{L}_f\right)\psi$, and $\overline{L_f \psi} = \bar{L}_f \overline{\psi}$, and in particular for real-valued functions $\psi$, we have $\overline{L_f \psi} = \bar{L}_f \psi$.

Furthermore, it is helpful to note that $L_f f = \bar{L}_f f = \Delta_f f = \Delta f - \abs{\nabla f}^2$, and so the $L^2$-gradient of $\H$ can be written in three equivalent ways as:

\begin{align*}
\DH(\varphi) & = -\left(\Delta_f + \lambda\,\textup{Id}\right)\left(f-\H\right) \\
& = -\left(L_f + \lambda\,\textup{Id}\right)\left(f-\H\right)\\
& = -\left(\bar{L}_f + \lambda\,\textup{Id}\right)\left(f-\H\right).
\end{align*}

All three of $L_f$, $\bar{L}_f$ and $\Delta_f$ are (complex) self-adjoint operators on $C^\infty(X, \C)$ with respect to the inner product:
\begin{equation}
\label{eq:WeightedInnerProduct}
(\psi_1, \psi_2)_f := \frac{1}{[\omega_0]^n} \int_X \psi_1 \bar{\psi}_2 e^{-f_\varphi}\omega^n_\varphi
\end{equation}
in a sense that $(L_f(\psi_1), \psi_2)_f = (\psi_1, L_f(\psi_2))_f$ and similarly for $\bar{L}_f$ and $\Delta_f$. Therefore, their eigenvalues are real.

By a standard argument (see e.g. \cite{F, PSSW}), it can be shown that when acting on the orthogonal complement of constant functions, both $-L_f$ and $-\bar{L}_f$ (and hence for $-\Delta_f$) have the lowest eigenvalue $\geq \lambda$. Due to its importance to our upcoming discussions, we state the result below and sketch its proof:

\begin{lemma}[c.f. \cite{F, PSSW}]
\label{lma:Eigenvalue}
Given any non-constant function $\psi \in C^{\infty}(X, \C)$ of $-L_f$ such that:
\[L_f \psi = -\mu \psi,\]
then we have $\mu \geq \lambda$, and equality holds if and only if $\nabla^{1,0}\bar{\psi} := g^{i\bar{j}}\frac{\partial\bar{\psi}}{\partial \bar{z}_j}\frac{\partial}{\partial z_i}$ is a holormophic vector field.
\end{lemma}

\begin{proof}[Sketch of Proof]
Given that $L_f \psi = -\mu \psi$ for some non-constant $\psi \in C^{\infty}(X, \C)$, we have:
\[\Delta\psi - g^{i\bar j}\nabla_{\bar j} f \, \nabla_i \psi = -\mu \psi.\]
Then by differentiating both sides by $\nabla_k$, we get:
\[\nabla_k\Delta\psi - g^{i\bar j}\nabla_k \left(\nabla_{\bar j}f \, \nabla_i \psi\right) = -\mu\nabla_k\psi.\]
Using the commutative formula for covariant derivatives and the fact that $\lambda g_{i\bar j} - R_{i\bar j} = \nabla_i \nabla_{\bar j}f$, one can conclude:
\begin{align*}
g^{i\bar j}\nabla_{\bar j}\nabla_k\nabla_i\psi - \lambda\nabla_k\psi - g^{i\bar j}\nabla_k\nabla_i\psi \cdot \nabla_{\bar j}f & = -\mu\nabla_k\psi\\
\Longrightarrow \quad g^{i\bar j}\nabla_{\bar j}\left(e^{-f}\nabla_k\nabla_i\psi\right) & = (\lambda - \mu)\nabla_k\psi.
\end{align*}
Finally, by multiplying both sides by $g^{k\bar l}\nabla_{\bar l}\bar{\psi}$, integrating both sides over $X$ with respect to the measure $e^{-f}\omega_\varphi^n$ and using integration-by-parts, we get:
\begin{equation}
\label{eq:PoincareInequality}
\int_X \abs{\nabla\nabla \psi}^2 e^{-f} \omega_\varphi^n = (\mu - \lambda)\int_X \abs{\nabla\psi}^2 e^{-f}\omega_\varphi^n
\end{equation}
where $\abs{\nabla\nabla\psi}^2 = g^{k\bar{l}}g^{i\bar{j}}\left(\nabla_k\nabla_i\psi\right)\left(\nabla_{\bar j}\nabla_{\bar l}\bar{\psi}\right)$. Therefore, we must have $\mu \geq\lambda$, and equality holds if and only if $\nabla_{\bar j}\nabla_{\bar l}\bar{\psi} = 0$ for any $j$ and $l$, which is equivalent to saying that $\nabla^{1,0}\bar\psi$ is holomorphic.
\end{proof}

If $\varphi$ is a critical point of $\H$, i.e. $\DH(\varphi) = 0$, then $f_{\varphi}$ satisfies the Euler-Lagrange's equation:
\[\left(L_f + \lambda\,\Id\right)\left(f-\H\right) = 0\]
or equivalently, $f_\varphi - \H(\varphi)$ is an eigenfunction of $L_f$ with eigenvalue $\lambda$. By \eqref{eq:PoincareInequality}, $\nabla f_\varphi$ is then real holomorphic. Therefore, the critical potentials $\varphi$ of $\H(\varphi)$ are those which $\omega_\varphi$ is a gradient K\"ahler-Ricci soliton.

The K\"ahler-Ricci flow:
\begin{equation}
\label{eq:KRF}
\frac{\partial\omega}{\partial t} = -\Ric(\omega) + \lambda\omega, \quad \omega(0) = \omega_0
\end{equation}
in the canonical class, i.e. $\lambda\omega_0 \in c_1(X)$, can be regarded as the flow of K\"ahler potentials:
\begin{align}
\label{eq:MongeAmpere}
\frac{\partial\varphi}{\partial t} & = f_\varphi - \H(\varphi)\\
\nonumber & = f_0 + \log\frac{\omega_\varphi^n}{\omega_0^n} + \lambda\varphi - \frac{1}{[\omega_0]^n}\int_X e^{-f_0 -\lambda\varphi}\omega_0^n - \H(\varphi)
\end{align}
in a sense that if $\varphi(t)$ satisfies \eqref{eq:MongeAmpere}, then $\omega(t) := \omega_0 + \sqrt{-1}\ddbar\varphi(t)$ satisfies \eqref{eq:KRF}.

\label{KRFMonotone} It is interesting to note that using $-\Delta_f \geq \lambda\,\Id$, we can show that $\H(\varphi)$ is monotonically increasing along the K\"ahler-Ricci flow $\frac{\partial\varphi}{\partial t} = f_\varphi - \H(\varphi)$:
\[\frac{d}{dt}\H(\varphi) = -\frac{1}{[\omega_0]^n}\int_X \left(f_\varphi - \H(\varphi)\right)\cdot (\Delta_f + \lambda\,\Id)\left(f_\varphi - \H(\varphi)\right)\,e^{-f_\varphi}\omega_\varphi^n \geq 0.\]
Here we have used the fact that $-\Delta_f - \lambda\,\Id \geq 0$ acting on non-constant functions and that $f_\varphi - \H(\varphi)$ is orthogonal to constant functions (see also \cite[Section 2]{He}).

\section{Commutator of $L_f$ and $\bar{L}_f$}
We will make use of the operators $L_f$ and $\bar{L}_f$ in the proof of the main theorem. It is important to note that in general $\bar{L}_f L_f \not= L_f \bar{L}_f$. In this section, we will compute the product $\bar{L}_f L_f$ and $L_f \bar{L}_f$ acting on scalar functions, and show that if $\nabla f$ is real holomorphic, then $\bar{L}_f$ and $L_f$ indeed commute with each other.
\begin{lemma}
\label{lma:Commutator}
For any $\psi \in C^{\infty}(X, \R)$, we have:
\begin{align}
\label{eq:barL_L}
\bar{L}_f L_f \psi & = \Delta\Delta\psi - 2\langle \nabla f, \nabla\Delta\psi \rangle - \nabla_{\bar j}\nabla_i\psi \cdot \left(\nabla_j \nabla_{\bar i} f - \nabla_j f \cdot \nabla_{\bar i} f\right)\\
\nonumber & \quad - e^f \nabla^{\bar i}\left(e^{-f}\nabla_{\bar i}\nabla_{\bar j}f \cdot \nabla^{\bar j}\psi\right)\\
\label{eq:L_barL}
L_f \bar{L}_f \psi & = \Delta\Delta\psi - 2\langle \nabla f, \nabla\Delta\psi \rangle - \nabla_{\bar j}\nabla_i\psi \cdot \left(\nabla_j \nabla_{\bar i} f - \nabla_j f \cdot \nabla_{\bar i} f\right)\\
\nonumber & \quad - e^f \nabla^i \left(e^{-f}\nabla_i\nabla_j f \cdot \nabla^j \psi\right)
\end{align}
If $\varphi$ is a critical point of $\H$ (so that $\nabla f_\varphi$ is real holomorphic), then we have $\bar{L}_f L_f = L_f \bar{L}_f$.
\end{lemma}
\begin{proof}
It suffices to show \eqref{eq:barL_L} only then \eqref{eq:L_barL} follows from conjugation.
\begin{align}
\label{eq:barL_L_1}
\bar{L}_f L_f \psi & = \bar{L}_f \left(\Delta \psi - \nabla^i f \cdot \nabla_i \psi\right)\\
\nonumber & = \underbrace{\Delta\Delta\psi - \nabla^{\bar j} f \cdot \nabla_{\bar j}\Delta\psi}_{\bar{L}_f\left(\Delta\psi\right)} - \bar{L}_f\left(\nabla^i f \cdot \nabla_i \psi\right)
\end{align}

For convenience, we use holormophic normal coordinates (with respect to $\omega_\varphi$) in the rest of computations.
\begin{align*}
& \bar{L}_f \left(\nabla^i f \cdot \nabla_i\psi\right)\\
& = \frac{1}{2} \nabla_j \nabla_{\bar j}\left(\nabla_{\bar i}f \cdot \nabla_i\psi\right) + \frac{1}{2} \nabla_{\bar j} \nabla_j\left(\nabla_{\bar i}f \cdot \nabla_i\psi\right)- \nabla_j f \cdot \nabla_{\bar j} \left(\nabla_{\bar i}f \cdot \nabla_i\psi\right)\\
& = \frac{1}{2}\nabla_j \left(\nabla_{\bar j}\nabla_{\bar i}f \cdot \nabla_i\psi + \nabla_{\bar i} f \cdot \nabla_{\bar j}\nabla_i \psi\right) + \frac{1}{2} \nabla_{\bar j} \left(\nabla_j \nabla_{\bar i} f \cdot \nabla_i \psi + \nabla_{\bar i} f \cdot \nabla_j\nabla_i\psi\right)\\
& \qquad - \nabla_j f \cdot \nabla_{\bar j}\nabla_{\bar i}f \cdot \nabla_i \psi - \nabla_j f \cdot \nabla_{\bar i} f \cdot\nabla_{\bar j}\nabla_i \psi\\
& = \frac{1}{2}\nabla_j \nabla_{\bar j}\nabla_{\bar i}f \cdot \nabla_i\psi + \frac{1}{2}\nabla_{\bar j}\nabla_{\bar i}f \cdot \nabla_j\nabla_i\psi + \frac{1}{2}\nabla_j \nabla_{\bar i} f \cdot \nabla_{\bar j}\nabla_i\psi + \frac{1}{2}\nabla_{\bar i} f \cdot \nabla_i \Delta\psi\\
& \qquad + \frac{1}{2}\nabla_{\bar j}\nabla_{\bar i}\nabla_j f \cdot \nabla_i \psi + \frac{1}{2}\nabla_j\nabla_{\bar i} f \cdot \nabla_{\bar j}\nabla_i \psi + \frac{1}{2}\nabla_{\bar j}\nabla_{\bar i} f \cdot \nabla_j\nabla_i\psi +\frac{1}{2}\nabla_{\bar i} f \cdot \nabla_{\bar j}\nabla_j\nabla_i \psi\\
& \qquad - \nabla_j f \cdot \nabla_{\bar j}\nabla_{\bar i}f \cdot \nabla_i \psi - \nabla_j f \cdot \nabla_{\bar i} f \cdot\nabla_{\bar j}\nabla_i \psi
\end{align*}
Grouping the 5th and 8th terms together, we get:
\begin{align*}
& \frac{1}{2}\nabla_{\bar j}\nabla_{\bar i}\nabla_j f \cdot \nabla_i\psi + \frac{1}{2}\nabla_{\bar i}f \cdot \nabla_{\bar j}\nabla_j\nabla_i\psi\\
& = \frac{1}{2}\nabla_{\bar j}\nabla_j\nabla_{\bar i} f \cdot \nabla_i\psi + \frac{1}{2}\nabla_{\bar i}f \cdot \nabla_{\bar j}\nabla_i\nabla_j\psi\\
& = \frac{1}{2}\left(\nabla_j\nabla_{\bar j}\nabla_{\bar i}f-R_{\bar{j} j \bar{i}k}\cdot\nabla_{\bar{k}}f\right) \cdot \nabla_i\psi + \frac{1}{2}\nabla_{\bar i}f \cdot \left(\nabla_i\nabla_{\bar j}\nabla_j\psi - R_{\bar{j}ij\bar{k}}\cdot \nabla_k\psi\right)\\
& = \frac{1}{2}\nabla_j\nabla_{\bar j}\nabla_{\bar{i}}f \cdot \nabla_i \psi + \frac{1}{2}\nabla_{\bar i}f \cdot \nabla_i\Delta\psi.
\end{align*}
The Riemann curvature terms cancel each other by the Bianchi identity.

Substituting it back in, we get:
\begin{align*}
& \bar{L}_f \left(\nabla^i f \cdot \nabla_i\psi\right)\\
& = \nabla_{\bar j}\nabla_i\psi \cdot \left(\nabla_j \nabla_{\bar i} f - \nabla_j f \cdot \nabla_{\bar i} f\right) + \nabla_{\bar i}f \cdot \nabla_i\Delta\psi\\
& \qquad + \nabla_j\nabla_{\bar j}\nabla_{\bar i}f \cdot \nabla_i\psi + \nabla_{\bar j}\nabla_{\bar i}f \cdot \nabla_j\nabla_i\psi - \nabla_j f\cdot \nabla_{\bar j}\nabla_{\bar i}f \cdot \nabla_i\psi.
\end{align*}
It is straight-forward to verify that the last three terms sum up to $e^f\nabla_j\left(e^{-f}\nabla_{\bar j}\nabla_{\bar i}f \cdot \nabla_i\psi\right)$. Combining with \eqref{eq:barL_L_1}, we proved \eqref{eq:barL_L}. 
\end{proof}

Furthermore, we define the weighted divergence $\Div_f$ by:
\begin{align*}
\Div_f X & = \frac{1}{2}e^f \left[\nabla_i(e^{-f}X^i) + \nabla_{\bar i}(e^{-f}X^{\bar i})\right] & \text{for any vector field $X$}\\
\Div_f \alpha & = \frac{1}{2} e^f \left[\nabla^i (e^{-f}\alpha_i) + \nabla^{\bar i}(e^{-f}\alpha_{\bar i}) \right] & \text{for any 1-form $\alpha$}\\
\left[\Div_f \eta\right]_B & = \frac{1}{2}e^f \sum_A \nabla^A (e^{-f}\eta_{AB}) & \text{for any symmetric 2-tensor $\eta$}
\end{align*}
Clearly we have $\Div_f\nabla\psi = \Delta_f\psi$ for any $\psi \in C^{\infty}(X)$. Using the weighted divergence, one can also express $\bar{L}_f L_f$ as the following:
\begin{lemma}
\label{lma:DivDiv}
Given any $\psi \in C^{\infty}(X, \R)$, we denote $\nabla\bar{\nabla}\psi = \psi_{i\bar j} \left(dz^i \otimes d\bar{z}^j + d\bar{z}^j \otimes dz^i\right)$ which is a symmetric real 2-tensor. Then we have:
\begin{equation}
\label{eq:DivDiv}
2\,\Div_f \Div_f \nabla\bar\nabla\psi = \Delta\Delta\psi - 2\langle \nabla f, \nabla\Delta\psi \rangle + \nabla_j\nabla_{\bar i}\psi \left(\nabla^j f \cdot \nabla^{\bar i} f - \nabla^j\nabla^{\bar i} f\right)
\end{equation}
and hence from \eqref{eq:barL_L}, we have:
\begin{equation}
\label{eq:LLDivDiv}
\bar{L}_fL_f\psi = 2\,\Div_f\Div_f\nabla\bar\nabla\psi - e^f \nabla^{\bar i}\left(e^{-f}\nabla_{\bar i}\nabla_{\bar j}f \cdot \nabla^{\bar j}\psi\right).
\end{equation}
\end{lemma}
\begin{proof}
It can be proved by straight-forward computations. We first compute:
\begin{align*}
\left[\Div_f\left(\nabla\bar\nabla\psi\right)\right]_j & = \frac{1}{2} e^f \nabla^{\bar i} \left(e^{-f}\nabla_{\bar i}\nabla_j \psi\right)\\
& = \frac{1}{2}\left(-\nabla^{\bar i} f \cdot \nabla_{\bar i}\nabla_j\psi + \nabla^{\bar i}\nabla_{\bar i}\nabla_j\psi\right)\\
& = \frac{1}{2}\left(-\nabla^{\bar i} f \cdot \nabla_{\bar i}\nabla_j\psi + \nabla_j\Delta\psi\right)
\end{align*}
and so we can derive:
\begin{align*}
\Div_f \Div_f\left(\nabla\bar\nabla\psi\right) & = \textup{Re} \left\{e^f \nabla^j \left(e^{-f} \left[\Div_f\left(\nabla\bar\nabla\psi\right)\right]_j\right)\right\}\\
& = \textup{Re}\left\{-\nabla^j f \cdot  \left[\Div_f\left(\nabla\bar\nabla\psi\right)\right]_j + \nabla^j \left[\Div_f\left(\nabla\bar\nabla\psi\right)\right]_j\right\}\\
& = \frac{1}{2}\textup{Re}\left\{\nabla_{\bar i}\nabla_j\psi \cdot \nabla^{\bar i} f \cdot \nabla^j f - \nabla^j f \cdot \nabla_j\Delta\psi - \nabla^{\bar i} f \cdot \nabla_{\bar i}\Delta\psi\right.\\
& \qquad\qquad \left.- \nabla^j \nabla^{\bar i} f \cdot \nabla_{\bar i}\nabla_j\psi + \Delta\Delta\psi\right\}\\
& = \frac{1}{2}\left(\nabla_j\nabla_{\bar i}\psi \left(\nabla^j f \cdot \nabla^{\bar i} f - \nabla^j\nabla^{\bar i} f\right) - 2\langle \nabla\Delta\psi, \nabla f\rangle + \Delta\Delta\psi\right)
\end{align*}
as desired for \eqref{eq:DivDiv}. Then, \eqref{eq:LLDivDiv} follows from \eqref{eq:barL_L}.
\end{proof}

\begin{remark}
It is also helpful to note that $2\,\Div_f\Div_f = \left(\nabla\bar\nabla\right)^*$, which is the adjoint of
\begin{align*}
\nabla\bar\nabla : C^\infty(X, \R) & \to \text{Sym}^2(X)\\
\psi & \mapsto \psi_{i\bar j} \left(dz^i \otimes d\bar{z}^j + d\bar{z}^j \otimes dz^i\right)
\end{align*}
with respect to the measure $e^{-f_\varphi}\omega_\varphi^n$.
\end{remark}

\section{Second Variation of $\H$}
Next we derive the second variation formula of $\H$, and show that every K\"ahler-Ricci soliton is linearly stable with respect to $\H$. We first show:
\begin{proposition}[Evolution Equation of $\DH$]
\label{prop:FirstVarDH}
The variation of $\DH$ along a 1-parameter family $\varphi(t)$ of K\"ahler potential in $\K$ such that $\frac{\partial\varphi}{\partial t} = \psi$ is given by:
\begin{align}
\label{eq:FirstVarDH}
\frac{\partial}{\partial t}\DH & = - 2\,\Div_f\Div_f\left(\nabla\bar\nabla\psi\right) - 2\lambda\,\Delta_f\psi - \lambda^2 \left(\psi - \frac{1}{[\omega_0]^n}\int_X \psi e^{-f}\omega^n\right)\\
\nonumber & \qquad + \frac{\lambda}{[\omega_0]^n}\int_X \psi \, \DH\,e^{-f}\omega^n
\end{align}

\end{proposition}
\begin{proof}
First note that $\DH = -\left(\Delta_f + \lambda\,\Id\right)\left(f - \H\right)$. We compute:
\begin{align*}
& \frac{\partial}{\partial t}\left(\Delta_f + \lambda\,\Id\right)\left(f - \H\right)\\
& = \frac{\partial}{\partial t}\left(\Delta f - \abs{\nabla f}^2 + \lambda \left(f - \H\right)\right)\\
& = -g^{i\bar q}g^{p\bar j} \cdot \nabla_p \nabla_{\bar q}\left(\frac{\partial\varphi}{\partial t}\right) \cdot \nabla_i \nabla_{\bar j} f + \Delta \left(\frac{\partial f}{\partial t}\right) + g^{i\bar q}g^{p\bar j} \cdot \nabla_p\nabla_{\bar q}\left(\frac{\partial\varphi}{\partial t}\right) \cdot \nabla_i f \cdot \nabla_{\bar j} f\\
& \qquad - 2\left\langle\nabla \left(\frac{\partial f}{\partial t}\right), \nabla f\right\rangle + \lambda\left(\frac{\partial f}{\partial t} - \frac{\partial\H}{\partial t}\right)\\
& = \nabla_i\nabla_{\bar j}\psi \left(\nabla^i f \cdot \nabla^{\bar j} f - \nabla^i \nabla^{\bar j} f\right) + \Delta\left(\frac{\partial f}{\partial t}\right) - 2\left\langle\nabla \left(\frac{\partial f}{\partial t}\right), \nabla f\right\rangle + \lambda\left(\frac{\partial f}{\partial t} - \frac{\partial\H}{\partial t}\right)
\end{align*}
From Lemma \ref{lma:VarF}, we have $\displaystyle{
\frac{\partial f}{\partial t} = \Delta\psi + \lambda\left(\psi - \frac{1}{[\omega_0]^n}\int_X \psi e^{-f}\omega^n\right)}$. Therefore:
\begin{align*}
& \frac{\partial}{\partial t}\left(\Delta_f + \lambda\,\Id\right)\left(f - \H\right)\\
& =	\nabla_i\nabla_{\bar j}\psi \left(\nabla^i f \cdot \nabla^{\bar j} f - \nabla^i \nabla^{\bar j} f\right) + \Delta\left(\Delta\psi + \lambda\left(\psi - \frac{1}{[\omega_0]^n}\int_X \psi e^{-f}\omega^n\right)\right)\\
& \qquad - 2\left\langle\nabla \left(\Delta\psi + \lambda\left(\psi - \frac{1}{[\omega_0]^n}\int_X \psi e^{-f}\omega^n\right)\right), \nabla f\right\rangle\\
& \qquad + \lambda\left(\Delta\psi + \lambda\left(\psi - \frac{1}{[\omega_0]^n}\int_X \psi e^{-f}\omega^n\right) - \frac{\partial\H}{\partial t}\right)\\
& = \underbrace{\Delta\Delta\psi - 2\langle\nabla f, \nabla\Delta\psi\rangle + \nabla_i\nabla_{\bar j}\psi \left(\nabla^i f \cdot \nabla^{\bar j} f - \nabla^i \nabla^{\bar j} f\right)}_{= 2\,\Div_f\Div_f\left(\nabla\bar\nabla\psi\right) \; \text{from Lemma \ref{lma:DivDiv}}}\\
& \qquad + 2\lambda\Delta\psi - 2\lambda\langle \nabla f, \nabla \psi\rangle + \lambda^2 \left(\psi - \frac{1}{[\omega_0]^n}\int_X \psi e^{-f}\omega^n\right) - \lambda\frac{\partial\H}{\partial t}\\
& = 2\,\Div_f\Div_f\left(\nabla\bar\nabla\psi\right) + 2\lambda\,\Delta_f\psi + \lambda^2 \left(\psi - \frac{1}{[\omega_0]^n}\int_X \psi e^{-f}\omega^n\right) - \lambda\frac{\partial\H}{\partial t}\\
& = 2\,\Div_f\Div_f\left(\nabla\bar\nabla\psi\right) + 2\lambda\,\Delta_f\psi + \lambda^2 \left(\psi - \frac{1}{[\omega_0]^n}\int_X \psi e^{-f}\omega^n\right) - \underbrace{\frac{\lambda}{[\omega_0]^n}\int_X \psi \, \DH\,e^{-f}\omega^n}_{\text{from Proposition \ref{prop:FirstVarH}}}
\end{align*}
It completes the proof of \eqref{eq:FirstVarDH}.
\end{proof}

\begin{proposition}[Second Variation of $\H$]
\label{prop:SecondVarH}
Let $\varphi(s,t)$ be a 2-parameter family of potentials in $\K$. Denote $\chi := \frac{\partial\varphi}{\partial s}$ and $\psi := \frac{\partial\varphi}{\partial t}$, then the second variation of $\H(\varphi)$ is given by:
\begin{align}
\label{eq:SecondVarH} & \quad \frac{\partial^2}{\partial s \partial t}\H(\varphi(s,t))\\
\nonumber & = \frac{1}{[\omega_0]^n}\int_X \frac{\partial^2\varphi}{\partial s \partial t}\,\DH(\varphi)\,e^{-f_\varphi}\omega_\varphi^n - \frac{\lambda}{[\omega_0]^n}\int_X \left(\chi - \underline{\chi}\right)\left(\psi - \underline{\psi}\right)\,\DH(\varphi)\,e^{-f_\varphi}\omega_\varphi^n\\
\nonumber & \quad -\frac{1}{[\omega_0]^n}\int_X \psi\left[2\,\Div_f\Div_f\nabla\bar\nabla\chi + 2\lambda\,\Delta_f\chi + \lambda^2\left(\chi - \underline{\chi}\right)\right] e^{-f_\varphi}\omega_\varphi^n
\end{align}
where $\displaystyle{\underline{\chi} := \frac{1}{[\omega_0]^n}\int_X \chi e^{-f_\varphi}\omega_\varphi^n}$ and $\displaystyle{\underline{\psi} := \frac{1}{[\omega_0]^n}\int_X \psi e^{-f_\varphi}\omega_\varphi^n}$ are the averages of $\chi$ and $\psi$ over $X$ with respect to the measure $e^{-f_\varphi}\omega_\varphi^n$.
\end{proposition}

\begin{proof}
To begin, we recall that:
\[\frac{\partial}{\partial t}\H(\varphi) = \frac{1}{[\omega_0]^n}\int_X \psi\,\DH(\varphi)\,e^{-f_\varphi}\omega_\varphi^n.\]
Next we differentiate both sides by $s$:
\begin{align*}
\frac{\partial^2}{\partial s \partial t}\H(\varphi(s,t)) & = \frac{1}{[\omega_0]^n}\int_X \left(\frac{\partial\psi}{\partial s}\,\DH(\varphi) + \psi\,\frac{\partial}{\partial s}\DH(\varphi)\right)e^{-f_\varphi}\omega_\varphi^n\\
& \quad - \frac{1}{[\omega_0]^n}\int_X \psi\,\DH(\varphi)\left(\frac{\partial f_\varphi}{\partial t} - \Delta\chi\right)e^{-f_\varphi}\omega_\varphi^n
\end{align*}
Recall from \eqref{eq:f_varphi} that:
\[\frac{\partial f_\varphi}{\partial s} = \Delta \chi + \lambda\left(\chi - \underline{\chi}\right).\]
From Proposition \ref{prop:FirstVarDH}, we also have:
\[\frac{\partial}{\partial s}\DH(\varphi) = - 2\,\Div_f\Div_f\left(\nabla\bar\nabla\chi\right) - 2\lambda\,\Delta_f\chi - \lambda^2 \left(\chi - \underline{\chi}\right) + \frac{\lambda}{[\omega_0]^n}\int_X \chi \, \DH(\varphi)\,e^{-f}\omega^n.\]
Substituting these two results back in, we get:
\begin{align*}
\frac{\partial^2}{\partial s \partial t}\H(\varphi(s,t)) & = \frac{1}{[\omega_0]^n}\int_X \frac{\partial^2\varphi}{\partial s \partial t}\,\DH(\varphi)\,e^{-f_\varphi}\omega_\varphi^n\\
& \quad - \frac{1}{[\omega_0]^n}\int_X \psi \left[2\,\Div_f\Div_f\left(\nabla\bar\nabla\chi\right) + 2\lambda\,\Delta_f\chi + \lambda^2 \left(\chi - \underline{\chi}\right)\right]\,e^{-f_\varphi}\omega_\varphi^n\\
& \quad + \frac{\lambda}{[\omega_0]^n} \int_X \underline{\psi} \, \chi \, \DH(\varphi)\,e^{-f}\omega^n - \frac{\lambda}{[\omega_0]^n}\int_X \psi\left(\chi - \underline{\chi}\right)\,\DH(\varphi)\,e^{-f_\varphi}\omega_\varphi^n.
\end{align*}
Finally, using the fact that:
\[\left(\psi - \underline{\psi}\right)\left(\chi - \underline{\chi}\right)\,\DH(\varphi) = \psi\left(\chi-\underline{\chi}\right)\,\DH(\varphi) - \underline{\psi}\,\chi\,\DH(\varphi) + \underline{\psi}\,\underline{\chi}\,\DH(\varphi)\]
and $\displaystyle{\int_X \DH(\varphi)\,e^{-f_\varphi}\omega_\varphi^n = 0}$, we have completed the proof of the proposition.

\end{proof}

\begin{corollary}
In particular, if $\DH(\varphi(0,0)) = 0$ (i.e. $\omega_0$ is a K\"ahler-Ricci soliton), then we have:
\begin{align}
\label{eq:SecondVarHKRS}
& \left.\frac{\partial^2}{\partial s \partial t}\right|_{(s,t) = (0,0)} \H(\varphi(s,t))\\
\nonumber & = -\frac{1}{[\omega_0]^n}\int_X \psi \left[2\,\Div_f\Div_f\nabla\bar\nabla\chi + 2\lambda\,\Delta_f\chi + \lambda^2\left(\chi - \underline{\chi}\right)\right] e^{-f_\varphi}\omega_\varphi^n
\end{align}
\end{corollary}

\section{$\H$-Stability of K\"ahler-Ricci Solitons}
In the study of functionals in geometric analysis, the second variation formula is often associated with notions of stability. In the previous section we have computed the second variation formula of $\H$. Motivated by the second variation, we introduce:
\begin{definition}[Stability Operator]
In view of Proposition \ref{prop:SecondVarH}, we define the \emph{stability operator} $\mathcal{S}_f : T_\varphi\K \to T_\varphi\K$ by:
\begin{equation}
\label{eq:StabilityOperator}
\mathcal{S}_f(\psi) := 2\,\Div_f\Div_f\nabla\bar\nabla\psi + 2\lambda\,\Delta_f\psi + \lambda^2\left(\psi - \frac{1}{[\omega_0]^n}\int_X \psi e^{-f_\varphi}\omega_\varphi^n\right)
\end{equation}
\end{definition}

As such, the second variation of $\mathcal{H}$ at a critical point $\omega_\varphi$ is given by:
\[\left.\frac{d^2}{dt^2}\right|_{t = 0}\H(\varphi + t\psi) = -\left(\psi, \mathcal{S}_f(\psi)\right)_f\]

Since the functional $\H$ is monotonically increasing along the K\"ahler-Ricci flow, we say a K\"ahler-Ricci soliton $\omega_\varphi$ is \emph{stable with respect to $\H$} (or simply $\H$-stable) if and only if $\left.\frac{d^2}{dt^2}\right|_{t=0}\H(\varphi + t\psi) \leq 0$ for any $\psi \in T_\varphi\K$. We are ready to give the proof of our main theorem that \emph{any} K\"ahler-Ricci soliton is stable in this sense.

\begin{theorem}[$\H$-Stability]
\label{thm:Stability}
Suppose $\varphi$ is a critical point of $\H$, i.e. $\omega_\varphi := \omega_0 + \sqrt{-1}\ddbar\varphi$ is a K\"ahler-Ricci soliton, then we have:
\[\left.\frac{d^2}{dt^2}\right|_{t=0}\H(\varphi + t\psi) \leq 0\]
for any $\psi \in T_\varphi \K$, and equality holds if and only if $\nabla\psi$ is real holomorphic.
\end{theorem}

\begin{proof}
In view of
\[\left.\frac{d^2}{dt^2}\right|_{t = 0}\H(\varphi + t\psi) = -\left(\psi, \mathcal{S}_f(\psi)\right)_f\]
when $\varphi$ is a critical point of $\H$, it suffices to show the stability operator $\mathcal{S}_f$ is non-negative definite on $\K$. Since $\mathcal{S}_f$ is self-adjoint with respect to the $(\cdot, \cdot)_f$ and $\mathcal{S}_f(c) = 0$ for any constant $c$, we have $(c, \mathcal{S}_f(\psi))_f = 0$ as well and so:
\[\left(\psi, \mathcal{S}_f(\psi)\right)_f = \left(\psi - \underline{\psi}, \mathcal{S}_f(\psi - \underline{\psi})\right)_f.\]

When $\varphi$ is a K\"ahler potential such that $\omega_\varphi$ is a K\"ahler-Ricci soliton, we have $\nabla_i\nabla_j f_\varphi = \nabla_{\bar i} \nabla_{\bar j} f_\varphi = 0$ for any $i$ and $j$. Thus, the last term of \eqref{eq:LLDivDiv} in Lemma \ref{lma:DivDiv} vanishes, and we have:
\[\bar{L}_fL_f\psi = 2\,\Div_f\Div_f\nabla\bar\nabla\psi\]
for any $\psi \in T_\varphi\K$, and so:
\begin{align*}
\mathcal{S}_f(\psi) & = \bar{L}_f L_f \psi + 2\lambda\Delta\psi - 2\lambda\langle\nabla\psi, \nabla f\rangle + \lambda^2\left(\psi - \underline{\psi}\right)\\
& = \bar{L}_f L_f \psi + \lambda\left(\bar{L}_f + L_f\right)\psi + \lambda^2\left(\psi - \underline{\psi}\right)\\
& = \left(\bar{L}_f + \lambda\,\Id\right)\left(L_f + \lambda\,\Id\right)\left(\psi - \underline{\psi}\right)
\end{align*}

Note that $\bar{L}_f + \lambda\,\Id \leq 0$ and $L_f + \lambda\,\Id \leq 0$, and that they are self-adjoint and commutative at $t=0$ (from Lemma \ref{lma:Commutator}), so they can be simultaneously diagonalized and the product $\left(\bar{L}_f+\lambda\,\Id\right)\left(L_f + \lambda\,\Id\right)$ is non-negative definite. Since $\mathcal{S}_f$ is self-adjoint with respect to the $(\cdot, \cdot)_f$ and $\mathcal{S}_f(c) = 0$ for any constant $c$, we have $(c, \mathcal{S}_f(\psi))_f = 0$ as well and so:
\[\left(\psi, \mathcal{S}_f(\psi)\right)_f = \left(\psi - \underline{\psi}, \mathcal{S}_f(\psi - \underline{\psi})\right)_f \geq 0\]
since $\left(\bar{L}_f+\lambda\,\Id\right)\left(L_f + \lambda\,\Id\right) \geq 0$. It completes the proof that:
\[\left.\frac{d^2}{dt^2}\right|_{t = 0}\H(\varphi + t\psi) = -\left(\psi, \mathcal{S}_f(\psi)\right)_f = -\left(\psi - \underline{\psi}, \mathcal{S}_f(\psi - \underline{\psi})\right)_f  \leq 0.\]
Equality holds if and only if
\[\left(\bar{L}_f + \lambda\,\Id\right)\left(L_f + \lambda\,\Id\right)\left(\psi - \underline{\psi}\right) = 0,\]
which is equivalent to the fact that $\nabla\psi$ is a real holomorphic vector field.
\end{proof}

\section{Relation with Perelman's Entropy}
\label{sect:Perelman}
In this section, we focus on the case where $c_1(X) > 0$, and $\omega_0$ is a K\"ahler metric such that $\lambda\omega_0 \in c_1(X)$ (where $\lambda > 0$). Recall that Perelman's $\mathcal{W}$-functional defined by:
\[\mathcal{W}(g, f, \tau) := \int_X \left[2\tau(R+\abs{\nabla f}^2) + f - 2n\right]  \frac{e^{-f}}{(4\pi\tau)^n} \omega_g^n.\]
By taking a suitable $\tau = \tau_0$ such that $[\omega_0]^n = (4\pi\tau_0)^n$, the Perelman's $\mu$-functional is defined by:
\[\mu(g) := \inf\left\{\mathcal{W}(g, f, \tau_0) : \int_X e^{-f} \omega_g^n = (4\pi\tau_0)^n\right\}\]
The first variation of $\mu$ is given by:
\[\left.\frac{d}{dt}\right|_{t=0} \mu(g+th) = \frac{1}{(4\pi\tau_0)^n}\int_X \left\langle h, \frac{1}{2}g - \tau_0\left(\Ric + \nabla^2 f_{\min}\right)\right\rangle e^{-f_{\min}} dV_g\]
where $f_{\min}$ is the minimizer such that $\mu(g) = \mathcal{W}(g, f_{\min}, \tau_0)$.
Therefore, $g$ is a critical metric of $\mu$ if and only if $g$ is a Ricci soliton satisfying:
\[\Ric(g) + \nabla^2 f_{\min} = \frac{1}{2\tau_0}g.\]
In our case we have $\lambda\omega_0 \in c_1(X)$, so it is necessary that $\tau_0 = \frac{1}{2\lambda}$.

Our goal in this section is to show that the Perelman's $\mu$-functional is concave at K\"ahler-Ricci solitons along the direction of complex Hessians of potential functions. This result was previously proved by Tian--Zhu in \cite{TZ08} using the second variation of $\mu(g+t\nabla\bar\nabla\psi)$. Many dynamical stability results of the K\"ahler-Ricci flow were established using this results, for instance \cite{TZ08, SW,Zhe,Z13}.

We are going to show that the Boltzmann's type entropy $\H(\varphi+t\psi)$ is an upper \emph{barrier} of $\mu(g+t\nabla\bar\nabla\psi)$ up to an addition of a constant, and they coincide at $t = 0$ if $\omega_\varphi$ is a K\"ahler-Ricci soliton. Therefore, if the second variation of $\H(\varphi+t\psi)$ is non-positive at $t = 0$, then so is the second variation of $\mu(g+t\nabla\bar\nabla\psi)$, thus giving a new proof to Tian--Zhu's result.

\begin{proposition}[c.f. \cite{TZ08}]
\label{prop:PerelmanStability}
Given a gradient K\"ahler-Ricci soliton $g$ on $X$ with K\"ahler form $\omega_\varphi = \omega_0 + \sqrt{-1}\ddbar\varphi$, we have:
\[\left.\frac{d^2}{dt^2}\right|_{t=0} \mu\left(g+t\nabla\bar\nabla\psi\right) \leq \left.\frac{d^2}{dt^2}\right|_{t=0} \H(\varphi+t\psi) \leq 0.\]
Furthermore, we have $\left.\frac{d^2}{dt^2}\right|_{t=0} \mu\left(g+t\nabla\bar\nabla\psi\right) = 0$ if and only if $\nabla\psi$ is a real holomorphic vector field.
\end{proposition}

\begin{proof}
For any $t \in (-\varepsilon, \varepsilon)$, by the definition of $\mu$, we have:
\[
\mu\left(g + t\nabla\bar\nabla\psi\right) \leq \mathcal{W}\left(g + t\nabla\bar\nabla\psi, f_{\varphi+t\psi}, \tau_0\right)
\]
where $f_{\varphi+t\psi}$ is the Ricci potential of $\omega_{\varphi+t\psi}$. By the definition of $\mathcal{W}$, we have:
\begin{align*}
& \mathcal{W}\left(g + t\nabla\bar\nabla\psi, f_{\varphi+t\psi}, \tau_0\right)\\
 & = \frac{1}{[\omega_0]^n}\int_X \left[2\tau_0\left(R+\abs{\nabla f_{\varphi+t\psi}}^2\right) + f_{\varphi+t\psi} - 2n\right] \, e^{-f_{\varphi+t\psi}} \omega_{\varphi+t\psi}^n\\
& = \frac{1}{[\omega_0]^n}\int_X \left[2\tau_0\left(n - \Delta f_{\varphi+t\psi} + \abs{\nabla f_{\varphi+t\psi}}^2\right) + f_{\varphi + t\psi} - 2n\right] \, e^{-f_{\varphi+t\psi}} \omega_{\varphi+t\psi}^n\\
& = 2n(\tau_0 - 1) + \H(\varphi+t\psi).
\end{align*}
Here we used the fact that $\int_X \left(-\Delta f + \abs{\nabla f}^2\right) e^{-f} \omega^n = \int_X \Delta(e^{-f}) \omega^n = 0$. Therefore, for any $t \in (-\varepsilon, \varepsilon)$, we have
\begin{equation}
\label{eq:UpperBarrier}
\mu\left(g_\varphi + t\nabla\bar\nabla\psi\right) \leq \mathcal{W}\left(g + t\nabla\bar\nabla\psi, f_{\varphi+t\psi}, \tau_0\right) = \H(\varphi + t\psi) + 2n(\tau_0 - 1).
\end{equation}

At $t = 0$, we have $g + t\nabla\bar\nabla\psi = g$ and the Ricci potential $f_\varphi$ coincides with the minimizer $f_{\min}$ such that $\mu(g) = \mathcal{W}(g, f_{\min}, \tau_0)$. Therefore, we have:
\begin{equation}
\label{eq:AtTimeZero}
\mu\left(g\right) = \mathcal{W}\left(g, f_\varphi, \tau\right) = \H(\varphi) + 2n(\tau_0 - 1).
\end{equation}
Combining \eqref{eq:UpperBarrier} and \eqref{eq:AtTimeZero}, we have shown that $\H(\varphi+t\psi) + 2n(\tau_0 - 1)$ is an upper barrier of $\mu(g+t\nabla\bar\nabla\psi)$ and that they are equal at $t = 0$. Therefore, we have:

\[\eval{\frac{d^2}{dt^2}}{t=0} \mu\left(g+t\nabla\bar\nabla\psi\right) \leq \eval{\frac{d^2}{dt^2}}{t=0} \H(\varphi+t\psi).\]
The proposition then follows easily from Theorem \ref{thm:Stability}.
\end{proof}

\begin{remark}
In \cite{TZ08}, the second variation of $\mu(g+t\nabla\bar\nabla\psi)$ computed at a shrinking K\"ahler-Ricci soliton (using the notations in this article) is given by:
\[\eval{\frac{d^2}{dt^2}}{t=0} \mu\left(g+t\nabla\bar\nabla\psi\right) = \left(\psi, \; \left(L_f+\bar{L}_f+\lambda\,\Id\right)^{-1}\bar{L}_fL_f(\bar{L}_f+\lambda\,\Id)(L_f+\lambda\,\Id)\psi\right)_f \]
which is non-positive since $L_f \leq -\lambda\,\Id$ and $\bar{L}_f \leq -\lambda\,\Id$.
\end{remark}

\bibliographystyle{amsplain}
\bibliography{citations}

\end{document}